\documentclass{iopart}

\usepackage{graphicx,cite,color}

\expandafter\let\csname equation*\endcsname\relax
  \expandafter\let\csname endequation*\endcsname\relax

\def\Xint#1{\mathchoice
   {\XXint\displaystyle\textstyle{#1}}%
   {\XXint\textstyle\scriptstyle{#1}}%
   {\XXint\scriptstyle\scriptscriptstyle{#1}}%
   {\XXint\scriptscriptstyle\scriptscriptstyle{#1}}%
   \!\int}
\def\XXint#1#2#3{{\setbox0=\hbox{$#1{#2#3}{\int}$}
     \vcenter{\hbox{$#2#3$}}\kern-.5\wd0}}

\def\dashint{\Xint-}

\usepackage{amsmath,amsthm,amssymb,amsfonts}

\newtheorem{theorem}{Theorem}
\newtheorem{lemma}[theorem]{Lemma}

\newcommand{\average}{{\mathchoice {\kern1ex\vcenter{\hrule
height.4pt width 8pt depth0pt}
\kern-11pt} {\kern1ex\vcenter{\hrule height.4pt width 4.3pt
depth0pt} \kern-7pt} {} {} }}

\newcommand{\calA}     {\mathcal{A}}

\newcommand{\calL}     {\mathcal{L}}

\newcommand{\ee}         {\mathbf{e}}

\newcommand{\MM}        {{\bf M}}
\newcommand{\mn}        {{\bf m}}

\newcommand{\N}{\mathbb{N}}

\newcommand{\R}         {\mathbb{R}}

\newcommand{ \rr}        {{\bf r}}

\newcommand{\Sp}        {\mathbb{S}}

\renewcommand{\th}        {\vartheta}

\newcommand{\vol}        {\Omega}

\newcommand{\Z}         {\mathbb{Z}}

\begin{document}

\title[One-dimensional N\'eel walls under applied external fields
]{One-dimensional N\'eel walls under applied external fields}

\author{Milena Chermisi\footnote{Present address: Fresenius Medical
    Care, Via Crema, 8, Palazzo Pignano, Province of Cremona, 26020,
    Italy} and Cyrill B. Muratov}

\address{Department of Mathematical Sciences, New Jersey Institute of
  Technology, Newark, NJ 07102, USA}

\ead{muratov@njit.edu}

\begin{abstract}
  We present a detailed analysis of one-dimensional N\'eel walls in
  thin uniaxial ferromagnetic films in the presence of an in-plane
  applied external field in the direction normal to the easy axis.
  Within the reduced one-dimensional thin film model, we formulate a
  non-local variational problem whose minimizers are given by
  one-dimensional N\'eel wall profiles. We prove existence, uniqueness
  (up to translations and reflections), regularity, strict
  monotonicity and the precise asymptotics of the decay of the
  minimizers in the considered variational problem.
\end{abstract}

\vspace{2pc}
\noindent \hspace{-2.5mm} \emph{Keywords}: magnetic domains, thin
films, non-local variational problems

\section{Introduction}

Ferromagnetic materials are at the heart of modern information storage
technology, whose need to keep up with the ever-growing amount of
digital data, currently in excess of $10^{21}$ bytes worldwide, is
readily apparent \cite{eleftheriou10}. This is why these materials
have attracted a huge degree of attention since the early days of the
digital age. The basic principle of magnetic storage relies on the
tendency of the electron spins in ferromagnetic materials to align
along certain preferred directions, giving rise to {\em magnetic
  domains} \cite{hubert}. Registering and manipulating the
magnetization orientation in a given domain is then used to read and
write discrete data encoded by the magnetization orientation in each
domain.

One common magnetic storage solution relies on the use of thin
uniaxial ferromagnetic films in which the magnetization vector prefers
to align along either direction of the easy magnetocrystalline axis in
the film plane
\cite{eleftheriou10,hubert,moser02,zhu08,slaughter09}. When the film
thickness becomes sufficiently small (under a few tens of nanometers),
the magnetization vector is constrained to lie almost entirely in the
film plane. In this situation magnetic domains in epitaxial
(monocrystalline) films usually consist of relatively large regions,
in which the magnetization vector is nearly constant and oriented in
the direction of one of the two possible directions along the easy
axis. These regions are separated by narrow transition regions, called
{\em domain walls}, in which the magnetization vector rapidly rotates
between the two orientations
\cite{hubert,dennis02,oepen91,allenspach94,desimone00}. One of the
most common wall types in such materials is the {\em N\'eel wall},
which separates two regions of opposite magnetization by an in-plane
rotation and is oriented along the easy axis to ensure zero net
magnetic charge. In real materials these walls are often pinned to the
material imperfections, and their motion determines magnetization
reversal under the action of applied magnetic fields \cite{hubert}.

Studies of N\'eel walls have a long and somewhat controversial history
(see the discussions in \cite{hubert,aharoni}), but at present the
structure of the N\'eel wall in very thin films appears to be rather
well understood on the basis of micromagnetic arguments
\cite{hubert,dietze61,collette64,riedel71,desimone00,garcia99,
  trunk01, garcia04,mo:jcp06}. The basic features of the predicted
one-dimensional N\'eel wall profiles had been verified experimentally
in \cite{berger92} (see also \cite{wong96,jubert04}). Rigorous
mathematical studies of the N\'eel walls are more recent and go back
to the work of Garc\'ia-Cervera \cite{garcia99,garcia04}, who
undertook some analysis of the associated one-dimensional variational
problems and performed extensive numerical studies of the energy
functional obtained by Aharoni from the full micromagnetic energy
after restricting the admissible configurations to profiles which
depend only on one spatial variable \cite{aharoni66}.  Melcher further
studied the minimizers of the same functional in the class of
magnetization configurations constrained to the film plane and
established symmetry and monotonicity of the energy minimizing
profiles connecting the two opposite directions of the easy axis
\cite{melcher03}. Using a further one-dimensional thin film reduction
of the micromagnetic energy introduced in \cite{mo:jcp06}, Capella,
Melcher and Otto outlined the proof of uniqueness of the N\'eel wall
profile and its linearized stability with respect to one-dimensional
perturbations \cite{capella07}. Stability of geometrically constrained
one-dimensional N\'eel walls with respect to large two-dimensional
perturbations in soft materials was demonstrated asymptotically in
\cite{desimone06}. More recently, $\Gamma$-convergence studies of the
one-dimensional wall energy in the limit of very soft films and in the
presence of an applied in-plane field normal to the easy axis were
undertaken in \cite{ignat08,ignat09}, and a rigorous derivation of the
effective magnetization dynamics driven by the reduced thin film
energy introduced in \cite{capella07} from the full three-dimensional
Landau-Lifshitz-Gilbert equation was presented in \cite{melcher10}.

In this paper, we perform a detailed variational study of the N\'eel
walls, understood as one-dimensional minimizers of the reduced thin
film micromagnetic energy, in uniaxial materials in the presence of an
applied in-plane magnetic field in the direction perpendicular to the
easy axis, extending previous results for N\'eel walls in the absence
of the applied field. We prove existence, uniqueness (up to
translations and reflections), regularity, strict monotonicity and the
precise decay of the energy minimizing wall profiles. Our variational
setting is slightly different from that adopted in the earlier works
and relies on the angle variable rather than the two-dimensional unit
vector representation of the magnetization. For this reason our proofs
differ in a few technical aspects from those of \cite{melcher03}. In
fact, one of the purposes of our work was to clarify some of the
arguments in the analyses of \cite{garcia04,melcher03,capella07}. In
particular, we spell out the details of the proof of uniqueness of
minimizers within our setting and fill in the missing argument for
proving strict monotonicity of the angle variable as the function of
coordinate, which is needed to establish stability of the N\'eel wall
profile in \cite{capella07}. We also establish the precise asymptotic
behavior of the N\'eel wall profiles at infinity, which is new even in
the case of zero applied field. Let us note that while in this paper
we are not concerned with the logarithmic tail of the N\'eel walls in
very soft materials, which was one of the main focuses of
\cite{garcia04,melcher03,garcia99}, our decay estimates could be made
quantitative in this regime to yield the intermediate asymptotics of
the N\'eel wall profile away from the core.

The rest of our paper is organized as follows. In Sec. \ref{sec:model}
we discuss the basic micromagnetic energy and derive the reduced
one-dimensional energy that describes the N\'eel walls in the applied
in-plane field oriented normally to the easy axis. Then in
Sec. \ref{s:main} we present the variational setting for our analysis
and state our main result. Sec. \ref{s:lems} contains a few auxiliary
results and Sec. \ref{s:proofs} contains the proof of the main
theorem. We also discuss some open problems at the very end of
Sec. \ref{s:proofs}.

\section{Model}
\label{sec:model}

In this paper we are interested in the analysis of the energy
minimizing magnetization configurations in thin uniaxial ferromagnetic
films of large extent with the easy axis in the film plane. We also
wish to include the effect of an applied in-plane field in the
direction normal to the easy axis. The starting point in the studies
of such systems is the energy functional, introduced by Landau and
Lifshitz, which leads to a non-convex, nonlocal variational
problem. The functional, written in the CGS units, is
\cite{landau35,hubert,aharoni,landau8}:
\begin{align}\label{E_M}
  E(\MM) =\frac{A}{2|M_s|^2}\int_\vol |\nabla \MM|^2 \, d^3r
  +\frac{K}{2|M_s|^2}\int_\vol \Phi(\MM) \, d^3r - \int_\vol {\bf
    H}_{\rm ext} \cdot \MM \, d^3r
  \notag \\
  + \frac{1}{2} \int_{\R^3}\int_{\R^3} \frac{\nabla \cdot \MM( \rr)
    \nabla \cdot \MM(\rr')}{| \rr-\rr'|}\, d^3 r \, d^3 r' +
  \frac{M_s^2}{2K} \int_\vol | \mathbf H_\mathrm{ext} |^2 \, d^3r .
\end{align}
Here $\Omega \subset \R^3$ is the domain occupied by the ferromagnetic
material, $\MM:\R^3\to\R^3$ is the magnetization vector that satisfies
$|\MM|=M_s$ in $\vol$ and $\MM=0$ in $\R^3\setminus \vol$, the
positive constants $M_s$, $A$ and $K$ are the material parameters
referred to as the saturation magnetization, exchange constant and the
anisotropy constant, respectively, $\mathbf H_\mathrm{ext}$ is an
applied external field, and $\Phi: \R^3 \to \R$ is a non-negative
potential that has several minima at which $\Phi$ vanishes. Note that
$\nabla \cdot \MM$ in the double integral is understood in the
distributional sense.

The micromagnetic energy in \eqref{E_M} is composed of five terms: the
exchange energy term which penalizes the spatial variations of the
magnetization $\MM$, the anisotropy term reflecting the
magnetocrystalline properties of the material, the Zeeman energy
favoring the alignment of $\MM$ with the applied external field,  the
stray-field energy, which is nonlocal and favors vanishing
distributional divergence, i.e., $\nabla \cdot \MM = 0$ both in
$\Omega$ and on $\partial \Omega$ (the so-called pole-avoidance
principle), and an inessential constant term added for convenience. In
the case of a uniaxial material of interest to us, there exists a
distinguished axis identified through a unit vector $\ee$, and $\Phi$
is given by $\Phi(\MM)=M_s^2- (\MM\cdot\ee)^2$, so that the minima of
$\Phi$ are $\{\pm \ee M_s\}$ \cite{hubert,aharoni,landau8}.

In the case of extended monocrystalline thin films with the in-plane
easy axis we have $\Omega = \R^2 \times (0,d)$, and without loss of
generality we may assume that $\ee = \ee_2$, where $\ee_i$ is the unit
vector in the $i$-th coordinate direction. For thin films (moderately
soft, ultra-thin) of practical interest to magnetic device
applications such as MRAMs (magnetoresistive random access memories)
\cite{dennis02,zhu08,slaughter09}, a significant reduction of the
energy in \eqref{E_M} is possible, giving rise to the {\em reduced
  thin film energy} \cite{desimone00,mo:jcp06}. To better explain the
relevant parameter regime, let us introduce the following quantities
\begin{align}
\ell=\bigg(\frac{A}{4\pi M_s^2}\bigg)^{1/2},
\qquad
L=\bigg(\frac{A}{K}\bigg)^{1/2},
\qquad
Q=\bigg(\frac{\ell}{L}\bigg)^{2},
\end{align}
called the exchange length, the Bloch wall thickness, and the material
quality factor, respectively. When the film is {\it ultra-thin} and
{\it soft}, we have $d \lesssim \ell \lesssim L$, but at the same time
one also has a balance $Ld / \ell^2 \sim 1$ for many materials
\cite{heinrich93}.  In this situation the dimensionless parameter
\begin{align}
\nu= \frac{4\pi M_s^2 d}{KL}
=\frac{Ld}{\ell^2}=\frac{d}{\ell \sqrt Q},
\end{align}
which is referred to as the {\em thin film parameter} \cite{mo:jcp06},
becomes a single measure of the strength of the magnetostatic
interaction relative to both anisotropy and exchange.

The reduced thin film energy is formally obtained from the full
micromagnetic energy in \eqref{E_M} by assuming that $\MM$ does not
vary in the direction of $\ee_3$ (the direction normal to the film),
setting the component of $\MM$ along $\ee_3$ to zero and passing to
the limit $Q \to 0$ and $d \to 0$ jointly, subject to $\nu = O(1)$
fixed, after rescaling lengths with $L$ \cite{mo:jcp06} (see also
\cite{desimone00}). Assuming further that $\mathbf H_\mathrm{ext} =
\mathbf e_1 h K / M_s$, after a suitable rescaling we arrive at the
following reduced energy functional:
\begin{align}
  \label{Em}
  E(\mn) = \frac12 \int_{\R^2} |\nabla \mn|^2 \, d^2 r & + \frac12
  \int_{\R^2} (\mn \cdot \mathbf e_1 - h)^2 \, d^2 r \notag \\
  & + {\nu \over 8 \pi} \int_{\R^2} \int_{\R^2} {\nabla \cdot \mn(\rr)
    \nabla \cdot \mn(\rr') \over | \rr - \rr'|} d^2 r \, d^2 r',
\end{align}
where now $\mn : \R^2 \to \Sp^1$ is the unit vector in the direction
of the magnetization in the film plane. Note that the assumptions on
$\MM$ used in this derivation are justified by the strong penalization
of the variations of $\MM$ across the normal direction to the film by
the exchange energy and by the strong penalization of the normal
component of $\MM$ by the shape anisotropy \cite{hubert}. Also, up to
a constant factor the last term in \eqref{Em} is simply the square of
the homogeneous $H^{-1/2}$-norm of $\nabla \cdot \mn$ in $\R^2$
\cite{desimone00}.

The reduced energy in \eqref{Em} is the starting point of the analysis
of the rest of our paper. Without loss of generality we may assume
that $h \geq 0$. Also note that for $h \geq 1$ the energy in
\eqref{Em} admits a unique global minimum $\mn = \mathbf e_1$ and no
N\'eel walls are, therefore, possible in this situation. For $h \in
[0, 1)$, on the other hand, there are two global minimizers $\mn_\pm =
(h, \pm \sqrt{1 - h^2})$ corresponding to the two monodomain
states. In the following, we will always assume that $h$ is in this
non-trivial range, in which N\'eel walls connecting the two states
appear. Let us point out that at the same time we do not allow the
external field to have a component in the direction of the easy axis,
since in this case only one monodomain state exists as the global
minimizer of the energy. Under an applied field in the direction of
the easy axis N\'eel walls begin to move, invading the domain with
higher energy density by the domain with the lower energy density
\cite{hubert,capella07}. Similarly, the considered wall orientation
along the easy axis is the only one that makes the stray field energy
of a one-dimensional profiles finite. When the wall makes a non-zero
angle with the easy axis (compare with \cite{mo:jap08}), it carries a
net magnetic charge, which makes the associated magnetostatic
potential for the wall in the whole of $\R^2$ infinite.

\section{Variational formulation and statement of the main result} 
\label{s:main}

We now turn to the study of one-dimensional N\'eel wall profiles. For
that we assume that $\mn$ varies only along $\ee_1$ and compute the
energy of such a configuration per unit length of the wall. It is
convenient to introduce the new variable $\vartheta = \vartheta(x)$
which gives the angle that the vector $\mn$ makes with $\mathbf e_2$
in the counter-clockwise direction as a function of the coordinate
along $\mathbf e_1$.  Thus, setting
\begin{align}
\text{$\mn(x)=(-\sin\th(x), \cos \th(x))\in \Sp^1$}
\end{align}
for every $x\in \R$, we can rewrite the one-dimensional version of the
functional in \eqref{Em} in terms of the angle variable $\th$ to
obtain the one-dimensional N\'eel wall energy (cf. \cite{mo:jcp06}):
\begin{equation}\begin{split}\label{1de}
    E (\th;\R) & :=\frac{1}{2}\int_\R \bigg\{ |\th_x|^2 + (\sin \th
    -h)^2 +\frac{\nu}{2} \sin \th \bigg(-\frac{d^2}{d
      x^2}\bigg)^{1/2}\sin\th \bigg\} \, dx \\
    &= \frac{1}{2}\int_\R \bigg( |\th_x|^2 +( \sin\th-h )^2 \bigg) dx
    + \frac{\nu}{8\pi} \int_\R\int_\R \frac{\big( \sin \th(x)- \sin
      \th(y)\big)^2}{ (x-y)^2}\, dx \, dy,
\end{split}\end{equation}
where, as usual, $\left( -{d^2 / dx^2} \right)^{1/2}$ denotes the
square root of the one-dimensional negative Laplacian (a linear
operator whose Fourier symbol is $|k|$), and, furthermore, we used the
identity \cite{dinezza12,mo:jap08}
\begin{equation}\label{12lap}
  \left( -{d^2 \over dx^2} \right)^{1/2} u(x)=\frac{1}{\pi}
  \dashint_{\R} \frac{u(x)-u(y)}{(x-y)^{2}}\, dy, 
\end{equation}
for every $x$ and, say, every $u\in C^\infty_c(\R)$, where $\dashint$
stands for the principal value of the integral.

We wish to study the minimizers of the energy in \eqref{1de} among the
profiles that connect the two distinct minima of the energy at $x =
\pm \infty$. To this end, we need to introduce a suitable admissible
class of functions which yields minimizers with the desired
properties. We propose to minimize $E(\th; \R)$ over the admissible
class
\begin{equation}\label{min1D}
  \calA:=\{\th\in H^1_{\rm loc}(\R): \, \th -\eta_h\in H^1(\R)\},
\end{equation}
where $\eta_h \in C^\infty(\R; [0, \pi])$ is a fixed non-increasing
function such that, setting
\begin{align}
  \theta_h:= \arcsin h \in [0, \tfrac{\pi}{2}),
\end{align} 
we have $\eta_h=\pi-\theta_h$ in $(-\infty, -1)$ and $\eta_h=\theta_h$
in $(1,+\infty)$.  We point out that the definition of $\calA$ does
not depend on the choice of $\eta_h$: If $\tilde\eta_h \in
C^\infty(\R; [0, 1])$ is a different non-increasing function, i.e.,
$\tilde\eta_h\neq \eta_h$, satisfying $\tilde\eta_h=\pi-\theta_h$ in
$(-\infty, -1)$ and $\tilde\eta_h=\theta_h$ in $(1,+\infty)$, then
\begin{align}
\calA=\{\tilde\th\in H^1_{\rm loc}(\R): \, \tilde\th -\tilde\eta_h\in H^1(\R)\}.
\end{align} 
Indeed, any $\th\in \calA$ satisfies also $\th-\tilde\eta_h
\in H^1(\R)$ for $\tilde\eta_h-\eta_h\in H^1(\R)$; vice versa, for the
same reason any $\tilde\th\in H^1_{\rm loc}(\R)$ with $\tilde\th
-\tilde\eta_h\in H^1(\R)$ belongs to $\calA$. Note that our choice of
the admissible class $\calA$ fixes the rotation sense of the N\'eel
wall, and the wall of the opposite rotation sense may be obtained from
the minimizer over $\calA$ by a reflection about $x = 0$.

It is easy to see that the Euler-Lagrange equation associated with the
functional in \eqref{1de} is given by
\begin{equation}\label{1del}
  -\th_{xx} + \cos\th \sin \th 
  - h\cos\th
  +\frac{\nu}{2} \cos\th \bigg(-\frac{d^2}{d x^2}\bigg)^{1/2} \sin\th 
  =0,
\end{equation}
with the boundary conditions at infinity
 \begin{equation}\label{limcon}
   \lim_{x\to +\infty } \th(x)=\theta_h, \qquad 
   \lim_{x\to -\infty}\th(x)=\pi-\theta_h.
 \end{equation}
The main result of this paper is the following.

\begin{theorem}[{\bf existence, uniqueness, regularity, strict
    monotonicity and decay of N\'eel walls}]\label{t-one} 
  For every $\nu > 0$ and every $h \in [0, 1)$ there exists a
  minimizer of $E(\th; \R)$ in \eqref{1de} over $\calA$ in
  \eqref{min1D}, which is unique (up to translations), strictly
  decreasing with range equal to $(\theta_h, \pi-\theta_h)$ and is a
  smooth solution of \eqref{1del} that satisfies the limit conditions
  given in \eqref{limcon}.  Moreover, if $\th^{(0)}\colon \R\to
  (\theta_h, \pi-\theta_h)$ is the minimizer of $E$ in the class
  $\calA$ satisfying $\th^{(0)}(0)= \tfrac{\pi}{2}$, then
  $\th^{(0)}(x) = \pi - \th^{(0)}(-x)$, and there exists a constant $c
  > 0$ such that $\lim_{x \to +\infty} x^2 (\th^{(0)}(x) - \theta_h) =
  c$.
\end{theorem}

\section{Some auxiliary lemmas}
\label{s:lems}

We start with a few preliminary considerations and lemmas. Let
$\th\in\calA$.  By Morrey's Theorem (see
\cite[Theorem~11.34]{Leoni_book}), $\th -\eta_h\in C^{1/2}(\R)$ and
$\th -\eta_h\to 0$ as $x\to\pm\infty$; that is, $\th\in C(\R)\cap
L^\infty(\R)$ and satisfies \eqref{limcon}. Furthermore, assuming in
addition that $\th(\R)\subset [\theta_h,\pi-\theta_h]$ and
$E(\th,\R)<+\infty$, and defining $\rho \colon \R\to [\theta_h, {\pi
  \over 2}]$ by
\begin{equation}\label{rho}
  \rho(x):=\left\{\begin{array}{lll}
      \th(x) &\qquad &\text{if } \th(x)\in [\theta_h, \frac{\pi}{2}], 
      \\
      \pi-\th(x)&\qquad &\text{if } \th(x)\in (\frac{\pi}{2}, \pi-\theta_h],
    \end{array}\right.
\end{equation}
for every $x\in \R$, we have $\sin \rho=\sin\th$, and since the map
$\th \mapsto \rho$ is Lipschitz, we also have $|\th_x|=|\rho_x|$
almost everywhere on $\R$.  Thus
\begin{align}\label{coest}
  +\infty> E( \th, \R) &= E(\rho, \R) = \frac{1}{2}\int_\R |\rho_x|^2
  \, dx + \frac{1}{2}\int_\R (\sin \rho-h )^2 \, dx +\frac{\nu}{4}
  \|\sin\rho\|_{\mathring{H}^{1/2}(\R)}^2
  \notag \\
  &\geq \frac{1}{2}\int_\R |\rho_x|^2 \, dx + \frac{c_h}{4}\int_\R
  \big(\rho-\theta_h \big)^2 \, dx
  +\frac{\nu}{4} \|\sin\rho\|_{\mathring{H}^{1/2}(\R)}^2 \notag \\
  &\geq \frac{c_h}{4} \big\|\rho-\theta_h \big\|_{H_1}^2
  +\frac{\nu}{4} \|\sin\rho\|_{\mathring{H}^{1/2}(\R)}^2,
\end{align}
where $c_h:=\cos^2 \big( \frac{\pi}{4}+ \frac{\theta_h}{2}\big) > 0$
for all $\theta_h <\frac{\pi}{2}$.  In the above inequality, we used
the fact that, since $\rho(\R)\subset [\theta_h, \frac{\pi}{2}]$,
$\theta_h \in [0, \frac{\pi}{2})$, and $\sin z\geq z/\sqrt2$ for all
$z\in [0, \frac{\pi}{4}]$, we have
\begin{align}
  \sin \rho-h =\sin\rho- \sin \theta_h = 2 \cos \bigg(
  \frac{\rho+\theta_h}{2}\bigg) \sin\bigg(
  \frac{\rho-\theta_h}{2}\bigg) \geq \frac{\sqrt{c_h}}{\sqrt 2}
  (\rho-\theta_h).
\end{align}

\begin{lemma}[{\bf restriction of rotations}]\label{p-uno}
  Let $\th\in \calA$ such that $E(\th)<+\infty$. Then there exists
  $\tilde\th\in \calA$ such that $\tilde\th(\R) \subset
  [\theta_h,\pi-\theta_h]$ and $E(\tilde\th)\leq E( \th)$, with strict
  inequality unless $\th(\R)\subset [\theta_h, \pi-\theta_h]$.
\end{lemma}

\begin{proof} {\bf Step~1}. We show first that there exists
  $\th^\tau\in \calA$ such that $\th^\tau(\R) \subset [0,\pi]$, $\sin
  \th^\tau=|\sin\th|$, and $E(\th^\tau)\leq E( \th)$.  Let
  $\th^\tau\colon \R\to [0, \pi]$ be defined by
\begin{equation}
\th^\tau(x):=\left\{\begin{array}{lll}
\th(x) -2k\pi&\qquad &\text{if } \th(x)\in [2k\pi, (2k+1)\pi)
\\
2k\pi-\th(x)&\qquad &\text{if } \th(x)\in [(2k-1)\pi, 2k\pi),
\end{array}\right.
\end{equation}
for every $x\in \R$.  The definition is well-posed since $\{[2k\pi,
(2k+1)\pi), [(2k-1)\pi, 2k\pi): k\in \Z\}$ is a partition of $\R$.
Notice that $\th^\tau$ is obtained by means of a translation by
$-2k\pi$ in each interval of the form $[2k\pi, (2k+1)\pi)$, and by
means of a reflection with respect to the origin and a translation by
$2k\pi$ in each interval of the form $[(2k-1)\pi, 2k\pi)$.  By
construction, $\th^\tau\in [0,\pi]$ and $\sin\th\leq |\sin\th| = \sin
\th^\tau$ so that $\|\sin\th\|_{L^2(\R)}=\|\sin\th^\tau\|_{L^2(\R)}$
and $-h \int_\R \sin \th^\tau dx \leq -h\int_\R\sin\th dx$, implying
that $\int_\R (\sin \th^\tau -h)^2 dx \leq \int_\R (\sin\th -h)^2
dx$. Furthermore, since the map $\th \mapsto \th^\tau$ is Lipschitz,
we also have
$\int_\R \th_x^2 dx =\int_\R\big(\th^\tau_x\big)^2 dx$. 

The conclusion now comes from the fact that
\begin{align}
  \int_\R u \bigg(-\frac{d^2}{d x^2}\bigg)^{1/2}u \, dx =
  \frac{1}{2\pi} \int_\R\int_\R \frac{\big( u(x)-u(y)\big)^2}{(x-y)^2}
  \, dx \, dy \hspace{2.5cm}
  \notag \\
  \geq \frac{1}{2\pi} \int_\R\int_\R \frac{\big(
    |u(x)|-|u(y)|\big)^2}{(x-y)^2} \, dx \, dy = \int_\R |u|
  \bigg(-\frac{d^2}{d x^2}\bigg)^{1/2}|u| \, dx.
\end{align}
 
\noindent {\bf Step~2}. Without loss of generality, in view of Step~1,
we can assume that $\th\in \calA$ is such that $\th(\R) \subset
[0,\pi]$ and $E(\th)<+\infty$.  Set $I:=\{ x\in \R :\, \th(x)\in [0,
\theta_h)\cup (\pi- \theta_h, \pi]\}$ and let $\tilde\th\colon \R\to
[\theta_h, \pi-\theta_h]$ be defined by
\begin{equation}
\tilde\th(x):=\left\{\begin{array}{lll}
\theta_h&\qquad &\text{if } \th(x)<\theta_h
\\
\th(x)&\qquad &\text{if } \th(x)\in [\theta_h, \pi -\theta_h]=\R\setminus I,
\\
\pi -\theta_h&\qquad &\text{if } \th(x)>\pi -\theta_h
\end{array}\right.
\end{equation}
for every $x\in \R$.
 
Notice that, since $\th=\tilde\th$ in $\R\setminus I$, we have
\begin{align}
  \| \sin \th-h\|_{L^2 (\R) }^2- \| \sin \tilde\th-h\|_{L^2 (\R) }^2 =
  \| \sin \th-h\|_{L^2 (I) }^2\geq 0.
\end{align}
Moreover, since $\tilde\th$ is constant on $I$, we have
\begin{align}
  \| \sin \tilde\th\|_{\mathring{H}^{1/2} (\R) }^2 = \frac{1}{2 \pi}
  \int_{\R\setminus I} \int_{\R\setminus I} \frac{
    \big(\sin\th(x)-\sin \th(y)\big)^2 }{(x-y)^2} \, dx \, dy \notag
  \\
  + \frac{1}{\pi} \int_{\R\setminus I} \int_{I} \frac{ \big(h-\sin
    \th(y)\big)^2 }{(x-y)^2} \, dx \, dy ,
\end{align}
and so
\begin{align}
  \| \sin \th\|_{\mathring{H}^{1/2} (\R) }^2- \| \sin
  \tilde\th\|_{\mathring{H}^{1/2} (\R) }^2 = \frac{1}{2 \pi} \int_I \int_I
  \frac{ \big(\sin\th(x)-\sin \th(y)\big)^2 }{(x-y)^2} \, dx \, dy
  \notag \\
  + \frac{1}{\pi}  \int_{\R\setminus I} \int_I \frac{ \big(\sin\th(x)-\sin
    \th(y)\big)^2 - \big(h -\sin \th(y)\big)^2 }{(x-y)^2} \, dx \, dy
  \\
  \geq \frac{1}{\pi}  \int_{\R\setminus I} \int_I \frac{ \big(\sin\th(x)-h\big)
    \big(h + \sin\th(x)-2\sin \th(y)\big) }{(x-y)^2} \, dx \, dy >0,
  \notag
\end{align}
where in the second to last inequality we have used the identity
$A^2-B^2=(A-B)(A+B)$ and in the last one the fact that for every $x\in
I$ and every $y\in \R\setminus I$ the following inequalities hold:
$\sin\th(x)< h \leq \sin\th(y)$.  Then we conclude that
$E(\tilde\th)<E(\th)$.
\end{proof}

The following rearrangement property is a consequence of Lemma~7.17 in
\cite{lieb-loss}.
\begin{lemma}\label{l-uno}
  Let $u \in L^2(\R)$ be a nonnegative function and let $u^*$ be its
  symmetric decreasing rearrangement, i.e.,
  \begin{align}u^*(x):= \int_0^{+\infty} \chi_{\{ u>t\}^* }(x) \, dt,
\end{align}
where for a Borel set $A\subset \R$ the rearranged set $A^*$ is the
interval with measure $\calL^1(A)$ centered at the origin. Then
\begin{align}
  \|u^*\|_{\mathring{H}^{1/2} (\R) }^2 = \int_\R
  u^*\bigg(-\frac{d^2}{d x^2}\bigg)^{1/2} u^* \, dx \leq \int_\R u
  \bigg(-\frac{d^2}{d x^2}\bigg)^{1/2} u \, dx 
  =\|u\|_{\mathring{H}^{1/2} (\R) }^2,
\end{align}
with equality 
only if $u$ is a translation of a symmetric decreasing function.
\end{lemma}

To prove our main result, we first need the following preliminary
lemma. The idea goes back to \cite{melcher03}, except that here our
variable is the angle function rather than the first component of the
magnetization vector. Let us point out that the rearrangement argument
of Lemma \ref{l-due} only yields the non-increasing property of $\th$
for the minimizer.  To prove strict decrease, one needs an additional
argument presented in Step~3 of the proof of the main theorem below.

\begin{lemma}[{\bf rearrangement}]\label{l-due}
  Let $\th\in \calA$ be such that $\th(\R)\subset
  [\theta_h,\pi-\theta_h]$ and $E(\th)<+\infty$.  Then there exists a
  function $\th^o(x)\colon \R\to[\theta_h, \pi-\theta_h]$ satisfying
  \eqref{limcon} and the following properties:
\begin{align}
  \th^o(0)=\tfrac{\pi}{2}, \quad \th^o(x)=\pi-\th^o(-x), \quad
  \th^o_x\leq 0 ~~ \text{on} ~~ \R \quad \text{and} \quad E(\th^o)\leq
  E(\th),
\end{align}
where the equality in the latter expression holds only if $\sin \th$
is a translation of a symmetric decreasing function.
\end{lemma}

\begin{proof}
  Let $\rho \colon \R\to [\theta_h, \tfrac{\pi}{2}]$ be defined as in
  \eqref{rho} for every $x\in \R$.  Then, from the discussion at the
  beginning of Sec. \ref{s:lems} we have $E(\th, \R)=E(\rho,
  \R)$. Now, define $\rho^{o}:\R\to [\theta_h, \tfrac{\pi}{2}]$ by
  setting
\begin{equation}
\rho^o(x):=
\theta_h+\big(\rho(x)-\theta_h\big)^*, 
\end{equation}
where given a function $f$, $f^*$ stands for the symmetric
rearrangement of $f$.  This implies that $\rho^o$ is even,
$(\rho^*)_x\leq 0$ on $\R^+$, and $\rho^o(x)\to \theta_h$ as $|x|\to
+\infty$.  Moreover, the level sets of $\rho^o$ are simply the
rearrangement of the level sets of $\rho$, i.e.,
\begin{align}
\{x: \, \rho^o(x)>t\} = \{x:\, \rho(x)>t\}^*.
\end{align}
A consequence of this is the equimeasurability of the functions
$\rho^o$ and $\rho$, i.e.,
\begin{align}
\calL^1\big(\{x: \, \rho^o(x)>t\} \big)
=\calL^1\big(\{x: \, \rho(x)>t\} \big)
\end{align}
for every $t>0$. This, together with the Layer Cake Representation
Theorem~1.13 in \cite{lieb-loss}, yields
\begin{align}
  \int_\R \phi\big(\rho^o(x)\big)\, dx = \int_\R \phi(\rho(x))\, dx
\end{align}
for every monotone, absolutely continuous function $\phi : [0, \infty)
\to [0, \infty)$ satisfying $\phi(0) = 0$. Choosing $\phi(z)= \left(
  \sin \left( \min(z, \tfrac{\pi}{2}) \right)-h \right)_+^2$, we see
that
\begin{equation}\label{rho-due}
  \int_\R \big(\sin  \rho^o(x) -h\big)^2 \, dx
  = \int_\R \big(\sin\rho -h\big)^2 dx.
\end{equation}

We recall now the following property of rearrangements of 
any given Borel measurable function $f:\R\to\R$ 
vanishing at infinity:
\begin{align}
(\Phi\circ |f|)^* =\Phi\circ f^*
\end{align}
for every $\Phi:\R^+\to\R^+$ nondecreasing.  Applying the above
property to $f:=\rho-\theta_h:\R\to [0, \tfrac{\pi}{2}-\theta_h]$ and
$\Phi: f \mapsto \sin (\theta_h +f)$, which is increasing for every $f
\in [0, \tfrac{\pi}{2}-\theta_h]$, we get
\begin{equation}
  (\sin\rho)^*= \sin \big( \theta_h + (\rho-\theta_h)^*\big)= \sin
  \rho^o. 
\end{equation}
In view of the above identity and using also Lemma~\ref{l-uno}, we
have
\begin{align}\label{rho-tre}
  \int_\R \sin\rho^o \bigg(-\frac{d^2}{d x^2}\bigg)^{1/2}\sin\rho^o \,
  dx =\int_\R(\sin\rho)^* \bigg(-\frac{d^2}{d
    x^2}\bigg)^{1/2}(\sin\rho)^* \, dx \hspace{1cm}
  \notag \\
  \leq \int_\R \sin\rho \bigg(-\frac{d^2}{d x^2}\bigg)^{1/2}\sin\rho
  \, dx = \int_\R \sin\th \bigg(-\frac{d^2}{d x^2}\bigg)^{1/2}\sin\th
  \, dx,
\end{align}
where the equality holds only if $\sin \th$ is a translation of a
symmetric decreasing function.  Finally, by Lemma~7.17 in
\cite{lieb-loss},
\begin{equation}\label{rho-uno}
  \int_\R |\rho^o_x|^2 dx =
  \int_\R\Big[\big(\rho(x)-\theta_h\big)^*_x\Big]^2 dx \leq \int_\R
  |\rho_x|^2 dx  =\int_\R |\th_x|^2 dx. 
\end{equation}

Define $\th^{o}:\R\to [\theta_h, \pi-\theta_h]$ by setting
\begin{equation}
  \th^o(x):=\left\{\begin{array}{lll}
      \rho^o(x) &\qquad &\text{if } x\geq 0
      \\
      \pi  -\th^{o}(-x)= \pi-\rho^o(x)&\qquad &\text{if } x<0.
\end{array}\right.
\end{equation}
Since $\sin\th^o=\sin\rho^o$, by \eqref{rho-due}, \eqref{rho-tre}, and
\eqref{rho-uno} we conclude.
\end{proof}

We now investigate the decay of monotone solutions of \eqref{1del}
satisfying \eqref{limcon}. This information, combined with the
properties of the fundamental solution of the linearization of
\eqref{1del} around $\theta_h$ (see, e.g., \cite[Sec. 5.1]{garcia04}),
will be used to establish the precise asymptotics behavior of the
minimizers of $E$ over $\calA$ as $x \to \pm \infty$.

\begin{lemma}
  \label{l:udu} Let $\th \in C^\infty(\R)$ be a non-increasing
  solution of \eqref{1del} satisfying \eqref{limcon}, with $\th(-x) =
  \pi - \th(x)$ for all $x \in \R$ and all derivatives vanishing as $x
  \to \pm \infty$. Let $u = \sin \th - h$ and assume that there exist
  $c > 0$ and $\alpha \in (0, 2]$ such that $\| u
  \|_{W^{2,\infty}(\mathbb R)} \leq c$ and
  \begin{align}
    \label{eq:ual}
    u(x) \leq {c \over 1 + |x|^\alpha} \qquad \forall x \in \R.
  \end{align}
  Then there exists $C = C(c, \alpha, h, \nu) > 0$ such that
  \begin{align}
    \label{eq:ualx}
    |u_x(x)| , |u_{xx}(x)|, \left| \left( -{d^2 \over dx^2}
      \right)^{1/2} u(x) \right| \leq {C \over 1 + |x|^\alpha} \qquad
    \forall x \in \R.
  \end{align}
\end{lemma}

\begin{proof}
  Throughout the proof we assume that $x > 0$ is sufficiently large
  depending only on $c$, $\alpha$ and $h$. All the constants in the
  estimates are also assumed to depend only on $c$, $\alpha$, $h$ and
  $\nu$. 

  Equation \eqref{1del}
  written in terms of $u$ reads
  \begin{align}
    \label{eq:elu}
    {u_{xx} \over 1 - h^2 - 2 u h - u^2} + {(h + u) u_x^2 \over (1 -
      h^2 - 2 u h - u^2)^2 } = u + {\nu \over 2} \left( -{d^2 \over
        dx^2} \right)^{1/2} u.
  \end{align}
  In particular, since $u(x)$ goes to zero as $x \to \infty$ together
  with all its derivatives, we also have that $\displaystyle \lim_{x
    \to \infty} M(x) = 0$, where
  \begin{align}
    \label{eq:M}
    M(x) := \max_{[x,
      \infty)} \left| \left( -{d^2 \over dx^2} \right)^{1/2} u
    \right|. 
  \end{align}
  Now, multiply \eqref{1del} by $\th_x$ and integrate over $(x,
  \infty)$. Together with the monotonicity of $u(x)$, this yields
  \begin{align}
    \label{eq:uxM}
    \th_x^2(x) \leq u^2(x) + \nu u(x) M(x).
  \end{align}
  Using the fact that $u_x = \th_x \cos \th$, from \eqref{eq:elu} and
  \eqref{eq:uxM} we obtain
  \begin{align}
    \label{eq:uxxM}
    |u_{xx}(x)| \leq C (u(x) + M(x) ).
  \end{align}

  On the other hand, for every $\delta > 0$ sufficiently small we have
  by \eqref{12lap}
  \begin{align}
    \label{eq:D12u}
    \left( -{d^2 \over dx^2} \right)^{1/2} u(x) = \frac1\pi
    \int_{-\infty}^{-\tfrac12 x} {u(x) - u(y) \over (x - y)^2} dy +
    \frac1\pi \int_{-\tfrac12 x}^{\tfrac12 x} {u(x) - u(y) \over (x -
      y)^2} dy
    \notag \\
    + \frac1\pi \int_{\tfrac12 x}^{x-\delta} {u(x) - u(y) \over (x -
      y)^2} dy + \frac1\pi \dashint_{x-\delta}^{x+\delta} {u(x) - u(y)
      \over (x - y)^2} dy + \frac1\pi \int_{x+\delta}^{+\infty} {u(x)
      - u(y) \over (x - y)^2} dy.
  \end{align}
  Using the symmetric decreasing property of $u(x)$, we can then
  estimate the left-hand side of \eqref{eq:D12u} as
  \begin{align}
    \label{eq:D12uu}
    \left| \left( -{d^2 \over dx^2} \right)^{1/2} u(x) \right| \leq C
    \left( \delta^{-1} u(\tfrac12 x) + \delta \max_{[x-\delta,
        x+\delta]} |u_{xx}| + \int_{-\tfrac12 x}^{\tfrac12 x} {u(y)
        \over (x - y)^2} dy \right),
  \end{align}
  where to estimate the fourth term in the right-hand side of
  \eqref{eq:D12u} we used Taylor formula $u(y) = u(x) + u_x(x) (y - x)
  + \tfrac12 u_{xx}(\bar x(y)) (y - x)^2$, for some $\bar x(y)$ lying
  between $x$ and $y$, and noted that the linear term in $y - x$ does
  not contribute to the principal value of the integral.  Combining
  the estimate in \eqref{eq:D12uu} with \eqref{eq:uxxM} and the
  assumption on the decay of $u$ in \eqref{eq:ual} then yields
  \begin{align}
    \label{eq:D12uuu}
    \left| \left( -{d^2 \over dx^2} \right)^{1/2} u(x) \right| \leq C
    \left( {\delta^{-1} \over x^\alpha} + \delta M(x - \delta) + {1
        \over x^2} + {1 \over x^{1+\alpha}} \right).
  \end{align}
  Now, choosing $\delta > 0$ sufficiently small, the last estimate
  implies
  \begin{align}
    \label{eq:D12ult}
    \left| \left( -{d^2 \over dx^2} \right)^{1/2} u(x) \right| \leq C
    x^{-\alpha} + \tfrac12 M(x - \delta),
  \end{align}
  for every $\alpha \leq 2$. Therefore, by monotonicity of $M(x)$, we
  also have 
  \begin{align}
    \label{eq:MM}
    M(x) \leq C x^{-\alpha} + \tfrac12 M(x - \delta),
  \end{align}
  for all $x > 0$ sufficiently large. 

  Let us show that \eqref{eq:MM} implies the same kind of bound as in
  \eqref{eq:ual} for $M(x)$. We use an induction argument. Let $x_n :=
  x_0 + n \delta$ and $M_n := M(x_n)$ for $n \in \mathbb N$ and some
  $x_0 > 0$ to be fixed shortly. Clearly, by \eqref{eq:MM} we have
  $M_1 \leq c x_1^{-\alpha}$ for $c = 4 C + \tfrac12 M(x_0) (x_0 +
  \delta)^\alpha$. We claim that if also $M_{n-1} \leq c
  x_{n-1}^{-\alpha}$, then \eqref{eq:MM} implies $M_n \leq c
  x_n^{-\alpha}$, provided that $x_0$ is chosen to be sufficiently
  large. Indeed, by \eqref{eq:MM} and the assumption $M_{n-1} \leq c
  x_{n-1}^{-\alpha}$ we have
  \begin{align}
    \label{eq:Mn}
    M_n \leq \left( C + {c x_n^\alpha \over 2 ( x_n - \delta)^\alpha }
    \right) x_n^{-\alpha} \leq (C + \tfrac34 c) x_n^{-\alpha} \leq c
    x_n^{-\alpha},
  \end{align}
  provided that $x_0 \geq 6 \delta$, and the claim follows. Once we
  established the bound for $M_n$, the estimate $M(x) \leq C / (1 +
  x^\alpha)$ for all $x \in \R$ follows by monotonicity and
  boundedness of $M(x)$.

  To conclude the proof of the lemma, we combine the estimate for
  $M(x)$ just obtained with \eqref{eq:uxM} and \eqref{eq:uxxM}.
\end{proof}

\section{ Proof of the main result}
\label{s:proofs}

We are now in a position to prove the main result of this paper,
Theorem~\ref{t-one}.

\smallskip

\noindent {\bf  Step~1: existence.}
Let $\{\th_n\}\subset \calA$ be a minimizing sequence, i.e.,
 \begin{align}
 \lim_{n\to +\infty} E(\th_n) = \inf\{ E(\th): \th\in\calA\}<+\infty.
 \end{align}
 By translation invariance we may assume that $\th_n(0)=\frac{\pi}{2}$ 
 and 
 by Lemma~\ref{l-due}  
 that each $\th_n$  satisfies
 \begin{equation}\label{thnpro}\begin{split}
&
 \text{ 
 $\th_n(\R)\subset [\theta_h, \pi-\theta_h]$,  \,
 $\th_n(x)=\pi-\th_n(-x)$, 
 \, 
   $\lim_{x\to+\infty}\th_n(x)=\theta_h$,
 \, 
 $(\th_n)_x\leq 0$ on $\R$.}
\end{split} \end{equation} For each $n$, let $\rho_n\colon \R\to
[0,\tfrac{\pi}{2}]$ be defined as in \eqref{rho} (after replacing
$\th$ with $\th_n$).  By \eqref{coest} we have
\begin{align}
  \big\|\rho_n-\theta_h \big\|_{H^1(\R)}^2 +
  \|\sin\rho_n\|_{\mathring{H}^{1/2}(\R)}^2 \leq C<+\infty.
\end{align}
In view of Banach-Alaoglu-Bourbaki's Theorem, there exist a
subsequence of $\{\rho_n\}$, not relabeled, a function $v \in H^1(\R)$
and a function $u\in H^{1/2}(\R)$ such that
\begin{equation}\label{weaH1H1/2}
  \text{ $\rho_n-\theta_h\rightharpoonup  v $ weakly in $H^1(\R)$ and
    $\sin\rho_n\rightharpoonup  u$ weakly in $H^{1/2}(\R)$.} 
\end{equation}

Let us fix $k\in \N$. By the compact embedding
$H^1(-k,k)\subset\subset L^2(-k,k)$, we may find a subsequence, not
relabeled, such that $\rho_n-\theta_h\to v $ strongly in $L^2(-k, k)$
and $\calL^1$-a.e. in $(-k,k)$.  Thus, $\rho_n\to \rho$
$\calL^1$-a.e. in $(-k,k)$, where $\rho:= v +\theta_h$.  Moreover,
since
\begin{align}
  \sup_n \|\sin\rho_n\|_{\mathring{H}^{1/2}(-k,k)}^2 = \sup_n \left(
    \frac{1}{2 \pi} \int_{-k}^k \int_{-k}^k \frac{\big( \sin
      \rho_n(x)- \sin \rho_n(y)\big)^2}{ (x-y)^2}\, dx \, dy \right)
  <+\infty,
\end{align}
it follows by the Fractional Compact Embedding (see for instance
Section~8.5 of \cite{lieb-loss})
that $\{\sin\rho_n-h\}$ is precompact in $L^2(-k, k)$, that is, up to
a subsequence, not relabeled, $\sin\rho_n\to u$ strongly in $L^2(-k,
k)$ and $\calL^1$-a.e. in $(-k,k)$.  Then by the uniqueness of the
limits, $u=\sin \rho$ $\calL^1$-a.e. in $(-k,k)$ for every $k$.

Finally, by the lower semicontinuity of the $L^2$ and $H^{1/2}$ norms
with respect to their weak convergences and Fatou Lemma applied to the
sequence $\{(\sin \rho_n- h )^2 \}$, we have
\begin{align}
  \frac{1}{2}\int_{-k}^{k}|\rho_x|^2\, dx +\frac{1}{2}\int_{-k}^{k}
  (\sin \rho- h )^2\, dx + \frac{\nu}{8\pi } \int_{-k}^k\int_{-k}^k
  \frac{\big( \sin \rho(x)- \sin \rho(y)\big)^2}{ (x-y)^2}\, dx \, dy
  \notag \\
  \leq \liminf_{n\to+\infty} \frac{1}{2}\int_{-k}^{k}|(\rho_n)_x|^2\,
  dx +\frac{1}{2}\int_{-k}^{k} (\sin \rho_n-h )^2\, dx
  \\
  + \frac{\nu}{8\pi } \int_{-k}^k\int_{-k}^k \frac{\big( \sin
    \rho_n(x) - \sin \rho_n(y)\big)^2}{ (x-y)^2}\, dx \, dy \leq
  \liminf_{n\to +\infty } E(\rho_n, \R). \notag
\end{align}
Applying Lebesgue's Monotone Convergence Theorem to the sequences $\{
\chi_{(-k, k)} [(\rho_x)^2+(\sin \rho- h )^2]\}$ and
$\Big\{\chi_{(-k,k)^2} \frac{( \sin \th(x)- \sin \th(y))^2}{
  (x-y)^2}\Big\}$, we then obtain
\begin{equation}\label{rhomin}
  E(\rho,\R)\leq \liminf_{n\to +\infty }  E(\rho_n, \R)= \liminf_{n\to
    +\infty } E(\th_n, \R).
\end{equation}

Given such a $\rho$, let us define $\th:\R\to \R$
by setting
\begin{equation}\label{th^zero}
  \th^{(0)}(x):=\left\{\begin{array}{lll}
      \rho(x) &\qquad &\text{if } x\geq 0,
      \\
      \pi-\rho(x)&\qquad &\text{if }x<0,
\end{array}\right.
\end{equation}
for every $x\in \R$. 
We claim that  $\th^{(0)}$ satisfies the following properties:
\begin{equation}\label{thpro}
  \th^{(0)}(0)=\tfrac{\pi}{2},
  \quad
  \th^{(0)}(x)=\pi-\th^{(0)}(-x),
  \quad \lim_{x\to+\infty}\th^{(0)}(x)=\theta_h,
  \quad 
  \th^{(0)}_x\leq 0 ~~ \text{on} ~~ \R.
\end{equation}
Since $\rho-\theta_h\in H^1(\R)$, by Morrey's Theorem we have
$\displaystyle\lim_{x \to +\infty} \th^{(0)}(x) = \lim_{x \to +\infty}
\rho(x) = \theta_h$ and $\th^{(0)}\in C(\R)$.  Finally, since weak
convergence in $H^{1}(\R)$ implies pointwise and uniform convergence
on compacts (see for instance Theorem~8.6 in \cite{lieb-loss}), by the
first of \eqref{weaH1H1/2} we have $\rho_n\to \rho$ locally uniformly
on $\R$.  Therefore, the following properties
\begin{align}
  \rho_n(0)= \tfrac{\pi}{2}, \quad \rho_n(x)=\rho_n(-x), \quad
  \rho_n(x_1)\geq \rho_n(x_2) ~~ \text{if}~~ 0\leq x_1\leq x_2,
\end{align}
are preserved when taking the limit as $n\to+\infty$ which in turn
implies by construction that the properties in \eqref{thpro} are
satisfied.

Since by construction $E(\th^{(0)},\R)=E(\rho,\R)$, by \eqref{rhomin}
and \eqref{thpro}, we conclude that $\th^{(0)}$ is a minimizer for $E$
in the class $\calA$.

\smallskip

\noindent {\bf Step~2: Euler-Lagrange equation.}  Since the first
variation of $E$ at any global minimizer of $E$ is zero, we have that
the minimizer $\th^{(0)}$ of $E$, as well as any other minimizer of
$E$, satisfies
\begin{align}\label{1dfv}
  0= \left. \frac{d}{dt} \right|_{t=0} E(\th^{(0)}+t \varphi, \R)
  =\int_\R \bigg( \varphi_x \th^{(0)}_x + \varphi \cos\th^{(0)}
  \sin\th^{(0)} - h\varphi \cos\th^{(0)} \notag \\
  +\frac{\nu}{2} \varphi \cos \th^{(0)}
  \bigg(-\frac{d^2}{d x^2}\bigg)^{1/2}\sin\th^{(0)}
  \bigg)dx,
\end{align}
for every $\varphi\in H^1(\R)$. 
In other words, $\phi=\th^{(0)}$ is a weak solution of the ordinary
differential equation
\begin{equation}\label{1del_b}
-\phi_{xx} + b(x) \cos\phi =0,
\end{equation}
where $b(x):=\sin\th^{(0)} - h +\tfrac{1}{2} \nu (-d^2/d x^2)^{1/2}
\sin\th^{(0)} $.  Since $\th^{(0)}_x\in L^2(\R)$ and the sine function
is Lipschitz, we have $(\sin\th^{(0)})_x= \th^{(0)}_x \cos\th^{(0)}\in
L^2(\R)$, and by the estimate
\begin{equation}\label{H1subsetH1/2}
  \bigg\| u+ \frac{\nu}{2}\left(-\frac{d^2}{d
      x^2}\right)^{1/2}  u \bigg\|_{L^2(\R)} 
  \leq C \|u\|_{H^1(\R)},
\end{equation}
applied to $u:= \sin \th^{(0)}-h$, we conclude that $b(x)\in L^2(\R)$.
By \eqref{1del_b}, the weak second derivative of $\th^{(0)}$ is given
by $\th^{(0)}_{xx}= b(x) \cos\th^{(0)}$ and is, hence, in $L^2(\R)$.
Therefore, $\th^{(0)}_x\in H^1(\R)$ and by Morrey's Theorem, together
with the fact that $\th^{(0)}\in C(\R)$, we have that $\th^{(0)}_x$ is
continuous and bounded on $\R$ as well.  Using the interpolation
inequality (here we follow the arguments of \cite{capella07})
\begin{align}
  \|\th^{(0)}_x\|_{L^4(\R)} \leq \|\th^{(0)}_x\|_{L^\infty(\R)}^{1/2}
  \|\th^{(0)}_x\|_{L^2(\R)}^{1/2}<+\infty,
\end{align}
we have $ (\sin\th^{(0)})_{xx}= \th^{(0)}_{xx} \cos\th^{(0)}
-(\th^{(0)}_x)^2 \sin\th^{(0)}\in L^2(\R)$.  Thus, by
\eqref{H1subsetH1/2} and the fact that $ (d/dx)(-d^2/dx^2)^{1/2} u =
(-d^2/dx^2)^{1/2} u_x$, readily verified via the Fourier
representation of the fractional Laplacian \cite{dinezza12}, we
conclude that $b(x )\in H^1(\R)$ and so by Morrey's Theorem $b(x )$ is
continuous and bounded.  This in turn implies that $\th^{(0)}_{xx}=
b(x) \cos\th^{(0)}\in L^\infty(\R)\cap C(\R)$, that is, $\th^{(0)}\in
C^2(\R)$ is a classical solution of \eqref{1del_b}. Finally,
bootstrapping these regularity arguments, we conclude that $\th^{(0)}
\in C^\infty(\R)$, and, moreover, all the derivatives are bounded
uniformly in $\R$. Together with boundedness of their $L^2$ norms, the
latter implies that all the derivatives of $\th^{(0)}$ vanish at
infinity.

\noindent {\bf Step~3: strict monotonicity.}  First we show that
$\th^{(0)}_x(0)<0$ employing the same argument as in \cite{capella07}.
Since $\th^{(0)}\in C^2(\mathbb R)$ is a classical solution of
\eqref{1del_b}, if we assume that $\th^{(0)}(0)=\tfrac{\pi}{2}$ and
$\th^{(0)}_x(0)=0$, the uniqueness theorem for ordinary differential
equations implies that $\th^{(0)}$ must be identically equal to
$\tfrac{\pi}{2}$, a contradiction.

Now we prove $\th^{(0)}_x<0$ on $\R^+$.  We argue by contradiction and
assume that there exist $\bar x>0$ such that $\th^{(0)}_x(\bar x)=0$.
But then also $\th^{(0)}_{xx}$ must be zero at $\bar x$, because
otherwise $\th^{(0)}$ will be either strictly convex or strictly
concave at $\bar x$, contradicting the fact that $\th^{(0)}_{x}\leq 0$
on $\R$.  Moreover, by differentiating the Euler-Lagrange equation
\eqref{1del} and taking into account that $\th_x(\bar x) =
\th_{xx}(\bar x) = 0$, we get
\begin{equation}\begin{split}
    \th^{(0)}_{xxx}(\bar x) = \frac{\nu}{2}\cos\th^{(0)}(\bar x)
    g(\bar x), \qquad g := \bigg( -\frac{d^2}{d x^2}\bigg)^{1/2}
    \Big(\th^{(0)}_x\cos \th^{(0)}\Big).
  \end{split}\end{equation}
Now we observe that $\cos \th^{(0)}(\bar x)>0$ because
$\th^{(0)}(\R^+)\subset [\theta_h, \tfrac{\pi}{2}]$, and, recalling
\eqref{12lap}, taking into account that $\th_x(\bar x) = \th_{xx}(\bar
x) = 0$, and noticing that the integral converges near $\bar x$ in the
usual sense, we have
\begin{align}
  g(\bar x) & = - \frac{1}{\pi}\int_\R \frac{ \th^{(0)}_x (y)\cos
    \th^{(0)}(y)}{ (\bar x-y)^2}\, dy \notag
  \\
  & = - \frac{1}{\pi} \int_{\R^+} \frac{ \th^{(0)}_x (y)\cos
    \th^{(0)}(y)}{ (\bar x-y)^2}\, dy - \frac{1}{\pi} \int_{\R^+}
  \frac{ \th^{(0)}_x (y)\cos \th^{(0)}(y)}{ (\bar x+y)^2}\, dy 
  >0,
\end{align}
where we have used the fact that $\th^{(0)}_x(-x)=\th^{(0)}_x(x)\leq
0$, and the inequality is strict on a set of positive measure.
Therefore, $ \th^{(0)}_{xxx}(\bar x)>0$, which implies that
$\th^{(0)}$ is increasing at $\bar x$, contradicting the fact that
$\th^{(0)}_{x}\leq 0$ on $\R$.  Hence, $\th^{(0)}$ has to be strictly
decreasing on $\R^+$.
 
Since $\th^{(0)}(x) = \pi -\th^{(0)}(-x)$ for every $x<0$, 
it follows that  $\th^{(0)}_x<0$ also  on $\R^-$.

\smallskip

\noindent {\bf Step~4: uniqueness (up to translations).}  It is clear that in
view of the translational invariance the function
$\th^{(x_0)}(x)=\th^{(0)}(x-x_0)$ with any $x_0\in\R$ still belongs to
$\calA$ and satisfies $\th^{(x_0)}(x_0)=\tfrac{\pi}{2}$.  To conclude,
we have to show that every minimizers of $E$ in $\calA$ is of the form
$\th^{(0)}(x-x_0)$ for some $x_0\in \R$. Our argument is related to
the strict convexity of the integrand in \eqref{1de} written as a
function of $u = \sin \th$ and its derivative noted in
\cite{capella07}.

We employ the strict monotonicity of minimizers, which implies that
for every minimizer there is a unique point at which $\th =
\tfrac{\pi}{2}$. Let $\th^{(1)}$ and $\th^{(2)}$ be two different
minimizers, which, after a suitable translation, satisfy
$\th^{(1)}(0)=\tfrac{\pi}{2}$ and $\th^{(2)}(0)=\tfrac{\pi}{2}$.
Define $\tilde \th(x) := \arcsin \, \Big(\frac{\sin \th^{(1)}(x)
  +\sin\th^{(2)} (x) }{2}\Big)$ for $x \geq 0$ and $\tilde \th(x) :=
\pi - \arcsin \, \Big(\frac{\sin \th^{(1)}(x) +\sin\th^{(2)} (x) }{2}
\Big)$ for $x < 0$ (the function $\sin \tilde\th$ is symmetric
decreasing by Step 1).  We claim that for all $x \not= 0$ we have
\begin{equation}\label{exchange_p} 
\big( \tilde\th_x \big)^2 
=\frac{\big(\th^{(1)}_x \cos \th^{(1)} +\th^{(2)}_x \cos\th^{(2)}\big)^2 }{4
- (\sin \th^{(1)} +\sin\th^{(2)})^2}
\leq \frac{\big( \th^{(1)}_x\big)^2 +\big( \th^{(2)}_x\big)^2}{2}.
\end{equation}
Once the claim is proved, we get that $\tilde\th_x\in L^2(\R)$ and,
hence, $\tilde \th \in \mathcal A$. Moreover, since the anisotropy and
the stray-field terms in the energy are quadratic in $\sin\th$, we get
$E(\tilde\th, \R)<\frac{1}{2}[ E(\th^{(1)}, \R)+E(\th^{(2)}, \R)]$,
which contradicts the minimality of $\th^{(1)}$ and $\th^{(2)}$.

Let us come to the proof of \eqref{exchange_p}.  Observe that by
  two-dimensional Cauchy-Schwarz inequality
  \begin{align}
    \label{exch1}
    \frac{\big(\th^{(1)}_x \cos \th^{(1)} +\th^{(2)}_x
      \cos\th^{(2)}\big)^2 }{4 - (\sin \th^{(1)} +\sin\th^{(2)})^2}
    \leq \frac{ \left( \big(\th^{(1)}_x \big)^2 +\big( \th^{(2)}_x
        \big)^2 \right) \left( \cos^2 \th^{(1)} + \cos^2 \th^{(2)}
      \right) }{4 - (\sin \th^{(1)} +\sin\th^{(2)})^2}.
  \end{align}
  On the other hand, we have
\begin{multline}
  \frac{ 2 \cos^2 \th^{(1)} + 2 \cos^2 \th^{(2)} }{4 - (\sin \th^{(1)}
    +\sin\th^{(2)})^2} = \frac{ 4 - 2 \sin^2 \th^{(1)} - 2 \sin^2
    \th^{(2)} }{4 - (\sin \th^{(1)} +\sin\th^{(2)})^2} \\ \leq \frac{
    4 - \sin^2 \th^{(1)} - \sin^2 \th^{(2)} - 2 \sin \th^{(1)} \sin
    \th^{(2)} }{4 - (\sin \th^{(1)} +\sin\th^{(2)})^2} =
  1. \label{exch2}
\end{multline}
Combing \eqref{exch2} with \eqref{exch1} then yeilds
\eqref{exchange_p}.

\smallskip

\noindent {\bf Step~5: decay.} We claim that by the results of the
previous steps the unique minimizer $\th^{(0)}$ of $E$ in $\calA$
satisfies the assumptions of Lemma \ref{l:udu} with $\alpha =
\tfrac12$. Indeed, $\th^{(0)}$ is a smooth decreasing solution of
\eqref{1del} satisfying \eqref{limcon} with all derivatives vanishing
at infinity and obeying the required symmetry property. Furthermore,
since $u = \sin \th^{(0)} - h \in L^2(\R)$ is symmetric decreasing, by
an elementary property of monotone functions (see, e.g., \cite[Lemma
A.IV]{berestycki83}) we have that $u(x) \leq |x|^{-1/2} \| u
\|_{L^2(\R)}$. Therefore, the conclusions of Lemma \ref{l:udu} apply
to $\th^{(0)}$. We now claim that this, in turn, implies the same kind
of estimates for $\rho(x) - \theta_h$, where $\rho$ is defined by
\eqref{rho} with $\th = \th^{(0)}$. Indeed, since $u_x = \th_x^{(0)}
\cos \th^{(0)}$ and $u_{xx} = \th_{xx}^{(0)} \cos \th^{(0)} -
(\th_x^{(0)})^2 \sin \th^{(0)}$, the estimates for the derivatives
follow from \eqref{eq:ualx}, and to obtain the estimate for $(-d^2 /
dx^2)^{1/2} \rho$, one can use \eqref{eq:D12uu} with $u$ replaced by
$\rho$.

We now rewrite \eqref{1del} in the following form:
\begin{align}
  \label{eq:1del2}
  L (\rho(x) - \theta_h) = f(x), \qquad f(x) := f_1(x) + f_2(x) +
  f_3(x),
\end{align}
where 
\begin{align}
  \label{eq:L}
  L := -{d^2 \over dx^2} + \frac12 \nu \cos^2 \theta_h \left( -{d^2
      \over dx^2} \right)^{1/2} + \cos^2 \theta_h
\end{align}
is a linear operator that can be viewed as a map from $H^2(\R)$ to
$L^2(\R)$, and 
\begin{align}
  f_1(x) & := \cos \theta_h (\cos \theta_h - \cos \rho(x)) (
  \rho(x) -   \theta_h) \notag \\
  & \qquad + \cos \rho(x) ( \cos \theta_h (\rho(x) -
  \theta_h) - \sin   \rho(x) + \sin \theta_h), \\
  f_2(x) & := \frac12 \nu \cos \theta_h \left( -{d^2 \over dx^2}
  \right)^{1/2} ( \cos \theta_h (\rho(x) - \theta_h) - \sin
  \rho(x) + \sin
  \theta_h  ), \\
  f_3(x) & := \frac12 \nu (\cos \theta_h - \cos \rho(x)) \left(
    -{d^2 \over dx^2} \right)^{1/2} (\sin \rho(x) - \sin
  \theta_h).
\end{align}
Note that $L$ represents the operator that generates the linearization
of \eqref{1del} around $\th = \theta_h$, and all the terms in the
definition of $f$ are of ``quadratic order'' in $\rho -
\theta_h$. Also note that the fundamental solution $G(x)$ associated
with $L$, i.e., the solution of $L G(x) = \delta(x)$, where
$\delta(x)$ is the Dirac delta-function, is a positive, continuous,
piecewise-smooth function with a jump of the derivative at the origin
and a decay $G(x) \sim |x|^{-2}$ at infinity (see Lemma
\ref{l:G}). Moreover, $L$ is invertible, and we have
\begin{align}
  \label{eq:Linv}
  \rho - \theta_h = L^{-1} f(x) = \int_\R G(x - y) f(y) dy.  
\end{align}

Observe that since a priori $u(x) \leq c / (1 + |x|^{1/2})$, using
Taylor expansion in $\rho - \theta_h$ we have $|f_1(x)| \leq C / (1 +
|x|^{1/2})^2$, and by Lemma \ref{l:udu} we also have $|f_3(x)| \leq C
/ (1 + |x|^{1/2})^2$. To prove that $|f_2(x)| \leq C / (1 +
|x|^{1/2})^2$ as well, we use the estimate in \eqref{eq:D12uu} once
again, taking into account that $|(\rho \cos \theta_h - \sin
\rho)_{xx} | \leq |\cos \theta_h - \cos \rho| \, |\rho_{xx}| +
|\rho_x|^2 \sin \rho \leq C / (1 + |x|^{1/2})^2$ for $x \not= 0$ and
arguing as in Lemma \ref{l:udu}. Thus, we have $|f(x)| \leq C / (1 +
|x|)$, and hence using \eqref{eq:Linv}, we obtain
\begin{align}
  \label{eq:gimprove}
  \rho(x) - \theta_h \leq \int_\R G(x - y) |f(y)| dy \leq C \int_\R
  {|f(y)| \over 1 + |x - y|^2} dy \leq {C' \over 1 + |x|}.
\end{align}

In view of \eqref{eq:gimprove}, we have now improved the estimate for
$u(x)$ to the one in \eqref{eq:ual} with $\alpha = 1$. Repeating the
argument above, we then conclude that $u(x) \leq C' / (1 +
|x|^2)$. Let us show that this estimate is, in fact, optimal, and
obtain the precise asymptotics of the decay of the profile. In view of
the estimate just mentioned and arguing as in Lemma \ref{l:udu}, we
have $|f(x)| \leq C / (1 + |x|^4)$. Therefore, using the multipole
expansion in \eqref{eq:Linv} and the decay property of $G(x)$ from
Lemma \ref{l:G}, we conclude that
\begin{align}
  \label{eq:multipole}
  \rho - \theta_h = {a \over |x|^2} + o(|x|^{-2}), \qquad a = {\nu
    \over 2 \pi \cos^2 \theta_h} \int_\R f(y) \, dy.
\end{align}
Clearly, $a \geq 0$ in \eqref{eq:multipole}, and to complete the proof
we need to show that $a > 0$. To establish this fact, we first note
that $\int_\R f_1(x) dx > 0$ and $\int_\R f_2(x) dx = 0$. Indeed, it
is easy to see that $f_1(x) > 0$ for all $x \in \R$, and the second
equality follows from the fact that the operator $(-d^2/dx^2)^{1/2}$
is self-adjoint and that constants lie in its kernel. To show that
$\int_\R f_3(x) dx > 0$, we use \eqref{12lap} and symmetrize the
integral to obtain
\begin{align}
  \label{eq:sdc}
  \int_\R f_3(x) \, dx & = -{\nu \over 4 \pi} \int_\R \int_\R { (\sin
    \rho(x) - \sin \rho(y)) (\cos \rho(x) - \cos \rho(y)) \over (x -
    y)^2} \, dx \,
  dy \notag \\
  & = {\nu \over 2 \pi} \int_\R \int_\R { \sin \big( \rho(x) + \rho(y)
    \big) \sin^2 \left( {\rho(x) - \rho(y) \over 2} \right) \over (x -
    y)^2} \, dx \, dy,
\end{align}
where we used trigonometric identities to arrive at the last line. In
view of the fact that $\rho \in (\theta_h, \tfrac12 \pi)$, the
right-hand side of \eqref{eq:sdc} is positive, and the claim follows.

This concludes the proof of the theorem. \hfill \qed

\smallskip

Let us remark that the arguments in the proof of uniqueness above,
with the test function $\th^t(x) := \arcsin (t \sin \th^{(1)}(x) + (1
- t) \sin \th^{(2)}(x))$ for $x \geq 0$ and $\th^t = \pi - \arcsin (t
\sin \th^{(1)}(x) + (1 - t) \sin \th^{(2)}(x))$ for $x < 0$, where $t
\in [0,1]$, could also be used to prove uniqueness of the {\em
  critical points} of the energy $E$ in the class of solutions of
\eqref{1del} with values in $(\th_h, \pi - \th_h)$, obeying
\eqref{limcon} and crossing the value of $\tfrac{\pi}{2}$ only once at
the origin (for a similar treatment see \cite{gm:nhm12}). However,
since the computations in this case become exceedingly tedious and the
precise behavior of the solutions at infinity may be needed, we have
not pursued this question any further in this paper. Nevertheless,
establishing such a uniqueness result would be helpful for
interpreting the results of the numerical solution of \eqref{1del} as
the N\'eel wall profiles. 

It would also be interesting to see if the one-dimensional N\'eel wall
profiles considered in this paper are the only minimizers (or even
critical points) of the two-dimensional thin film energy in \eqref{Em}
with respect to perturbations with compact support that have the
asymptotic behavior given by \eqref{limcon}. We note that in the case
$\nu = 0$ this problem reduces to the famous problem of De Giorgi,
whose solution in two space dimensions is now well understood
\cite{ghoussoub98} (see also \cite{delpino12} for a recent
overview). Whether such a result remains valid in the presence of a
non-local term ($\nu > 0$) remains to be seen (one result in this
direction was obtained in \cite{desimone06}). Let us note that while
for the local problem a continuous family of solutions obtained by
rotations of the one-dimensional profile exists, in the non-local
problem the orientation of the wall is expected to be fixed by the
condition that the net charge across the wall be zero. The latter only
allows walls that are parallel to the easy axis.

\ack 

This work was supported by NSF via grants DMS-0908279 and
DMS-1313687. The authors are grateful to H. Kn\"upfer for valuable
discussions and to the anonymous referee for a suggestion that
simplified the proof of \eqref{exchange_p}.

\begin{appendix}

\section{}
\label{sec:appendix}

\setcounter{theorem}{0}
\renewcommand{\thetheorem}{\Alph{section}.\arabic{theorem}}

The following lemma establishes the basic properties of the
fundamental solution for the operator $L$ (see also
\cite[Sec. 5.1]{garcia04}).

  \begin{lemma}
    \label{l:G}
    Let $G(x)$ be the fundamental solution for the operator $L$
    defined in \eqref{eq:L}. Then
    \begin{align}
      \label{eq:G}
      G(x) = {2 \nu \over \pi} \int_0^\infty {t e^{- t |x| \cos
          \theta_h} \over \nu^2 t^2 \cos^2 \theta_h + 4 (t^2 - 1)^2}
      \, dt.
    \end{align}
    In particular, $G \in C^\infty(\R \backslash \{0\}) \cap
    \mathrm{Lip}(\R) \cap L^\infty(\R) \cap L^1(\R)$, $G > 0$, $G(x) =
    G(-x)$, and
    \begin{align}
      \label{eq:Gdecay}
     G(x) = {\nu \over 2 \pi \cos^2
        \theta_h} |x|^{-2} + O (|x|^{-4}). 
    \end{align}
  \end{lemma}
  
  \begin{proof}
    The proof is a simple application of Fourier transform and contour
    integration techniques, which we present here for
    completeness. Observe first that the Fourier transform $\widehat
    G(k) = \int_\R e^{-i k x} G(x) \, dx$ of $G(x)$ is well-defined
    and is given by
    \begin{eqnarray}
      \label{eq:Gk}
      \widehat G(k) = {1 \over |k|^2 + \tfrac12 \nu \cos^2 \theta_h |k| +
        \cos^2 \theta_h}. 
    \end{eqnarray}
    Interpreting $|k| = \sqrt{k^2}$ as an analytic function of $k$ in
    the complex plane with a branch cut on the imaginary axis, i.e.,
    defining $\sqrt{(x + i y)^2} = |x| + i y \, \mathrm{sgn} \, x$ for
    $x \not= 0$, we can write the formula for inverting the Fourier
    transform of $G$ as
    \begin{align}
      \label{eq:Gx}
      G(x) = {1 \over 2 \pi} \int_\R {e^{i k x} \over k^2 + \tfrac12
        \nu \cos^2 \theta_h \sqrt{k^2} + \cos^2 \theta_h} \, dk,
    \end{align}
    and treat it as an integral along the real axis in the complex
    plain. In particular, $G(x)$ is even, and in the following it
    suffices to consider only $x \geq 0$.

    It is easy to see that with our choice of the analytic branch the
    integrand in \eqref{eq:Gx} has no poles. Furthermore, since the
    integral in \eqref{eq:Gx} over a semicircle of radius $R > 0$ in
    the positive imaginary half-plane vanishes for $x \geq 0$ as $R
    \to \infty$, we can deform the contour of integration to run back
    and forth along the positive imaginary axis. After some algebra,
    we then find that the integral in \eqref{eq:Gx} coincides with
    that in \eqref{eq:G}.

    From the representation in \eqref{eq:G}, one immediately concludes
    that $G(x)$ is positive, bounded, smooth for all $x \not=0$ and
    Lipschitz-continuous at $x = 0$. The estimate in \eqref{eq:Gdecay}
    is then obtained by an elementary asymptotic analysis of the
    integral in \eqref{eq:G}.
  \end{proof}
\end{appendix}






\section*{References}

\bibliographystyle{iopart-num}
\bibliography{biblioicti}

\end{document}